\documentclass[11pt]{amsart}

\usepackage[a4paper,hmargin=3.5cm,vmargin=4cm]{geometry}
\usepackage{amsfonts,amssymb,amscd,amstext}
\usepackage[dvips]{epsfig}
\usepackage{color}

\usepackage{fancyhdr}
\usepackage{verbatim}
\usepackage[utf8]{inputenc}

\input xy
\xyoption{all}

\usepackage{times}
\usepackage{enumerate}
\usepackage{titlesec}
\usepackage{mathrsfs}

\titleformat{\section}
{\filcenter\bfseries\large} {\thesection{.}}{0.2cm}{}
\titleformat{\subsection}[runin]
{\bfseries} {\thesubsection{.}}{0.15cm}{}[.]
\titleformat{\subsubsection}[runin]
{\em}{\thesubsubsection{.}}{0.15cm}{}[.]
\usepackage[up,bf]{caption}

\newtheorem{theorem}{Theorem}[section]
\newtheorem{proposition}[theorem]{Proposition}

\newtheorem{lemma}[theorem]{Lemma}

\theoremstyle{definition}
\newtheorem{definition}[theorem]{Definition}
\newtheorem{remark}[theorem]{Remark}

\numberwithin{equation}{section}
\numberwithin{figure}{section}

\newcommand\D{\overline{\mathbb D}}

\renewcommand\D{\mathbb D}

\newcommand\igot{\mathfrak{i}}

\renewcommand\igot{\mathfrak{i}}

\renewcommand\imath{\igot}

\begin{document}

\fancyhead[LO]{Curvature of minimal graphs}
\fancyhead[RE]{ D. Kalaj}
\fancyhead[RO,LE]{\thepage}

\thispagestyle{empty}

\vspace*{1cm}
\begin{center}
{\bf\LARGE   On a sharp form of curvature conjecture for  minimal graphs}

\vspace*{0.5cm}

{\large\bf  David Kalaj }
\end{center}


\vspace*{1cm}

\begin{quote}
{\small
\noindent {\bf Abstract.}\hspace*{0.1cm}
Recently, the author and Melentijevi\'c resolved the longstanding Gaussian curvature problem (\cite{kalaj2025gaussian}) by establishing a sharp inequality
\(
|\mathcal{K}| < c_0 = \frac{\pi^2}{2}
\)
for minimal graphs defined over the unit disk, evaluated at the point on the graph lying directly above the origin. The constant $c_0$ is called Heinz constant. Using this result, we derive an improved estimate for the Hopf constant $ c_1 $. Furthermore, we prove that for any given unit normal $ \mathbf{n} $, there always exists a minimal graph over the unit disk, bending at coordinate directions, whose Gaussian curvature at the point above the origin is smaller—but arbitrarily close to—the Gaussian curvature of the corresponding Scherk-type surface associated with this normal, lying above a bicentric quadrilateral. This sharp inequality strengthens a classical result of Finn and Osserman, which holds in the specific case where the unit normal is $(0,0,1)$.

\noindent{\bf Keywords}\hspace*{0.1cm} Conformal minimal surface, minimal graph, curvature}

\vspace*{0.2cm}


\vspace*{0.1cm}

\noindent{\bf MSC (2010):}\hspace*{0.1cm} 53A10, 32B15, 32E30, 32H02

\vspace*{0.1cm}
\noindent{\bf Date: \today} 
\end{quote}

\vspace{0.2cm}

\section{Introduction}

A minimal graph is a surface in $\mathbb{R}^3$ that locally minimizes its area. Such a surface can be represented as the graph of a function $\mathbf{f}: \Omega \to \mathbb{R}$, where $\Omega$ is a domain in $\mathbb{R}^2$. The surface is given by the set of points $(u, v, \mathbf{f}(u, v))$, $(u,v)\in\Omega$. The condition for the graph to be minimal is that its mean curvature is zero. This leads to the minimal surface equation:
$$
\text{div} \left( \frac{\nabla \mathbf{f}}{\sqrt{1 + |\nabla \mathbf{f}|^2}} \right) = 0
$$
which can be written out explicitly as:
$$
(1+\mathbf{f}_v^2)\mathbf{f}_{uu} - 2\mathbf{f}_u \mathbf{f}_v \mathbf{f}_{uv} + (1+\mathbf{f}_u^2)\mathbf{f}_{vv} = 0
$$
In this paper, we focus on minimal graphs defined over the unit disk $\mathbb{D} = \{(u,v) \in \mathbb{R}^2 : u^2 + v^2 \leq 1\}$.

The Gaussian curvature, $\mathcal{K}$, of a surface is a measure of its intrinsic curvature. It can be expressed in terms of the function $\mathbf{f}$ and its derivatives. The first fundamental form of the surface is given by:
$$
I = E \, du^2 + 2F \, du \, dv + G \, dv^2
$$
where $E = 1+\mathbf{f}_u^2$, $F = \mathbf{f}_u \mathbf{f}_v$, and $G = 1+\mathbf{f}_v^2$. The second fundamental form is given by:
$$
II = L \, du^2 + 2M \, du \, dv + N \, dv^2
$$
where $L = \frac{\mathbf{f}_{uu}}{W}$, $M = \frac{\mathbf{f}_{uv}}{W}$, $N = \frac{\mathbf{f}_{vv}}{W}$, and $W = \sqrt{1+\mathbf{f}_u^2+\mathbf{f}_v^2}$.
The Gaussian curvature is defined as $\mathcal{K} = \frac{LN-M^2}{EG-F^2}$. Substituting the expressions for $E, F, G, L, M, N$, we get:
$$
\mathcal{K} = \frac{\mathbf{f}_{uu}\mathbf{f}_{vv}-\mathbf{f}_{uv}^2}{(1+\mathbf{f}_u^2+\mathbf{f}_v^2)^2}
$$
Since the surface is minimal, its mean curvature $H = \frac{1}{2}\frac{EN-2FM+GL}{EG-F^2} = 0$. This implies that $L(1+\mathbf{f}_v^2) - 2Mf_uf_v + N(1+\mathbf{f}_u^2) = 0$, which is exactly the minimal surface equation mentioned above.

Notice that, in a graph the condition, which will be used in our results $\mathbf{f}_{uv}(0,0)=0$, means that at the origin, the local coordinate system formed by the $u$ and $v$ axes is already aligned with the principal directions. This simplifies calculations and provides a clear geometric interpretation: the surface is bending maximally in one coordinate direction and minimally (with opposite sign curvature) in the other, and these directions are aligned with the chosen $u$ and $v$ axes.  We say that such a graph is \emph{bending in coordinate directions at zero}. Notice that every minimal graph can be transformed into such a graph by an appropriate rotation around the $z$-axis.

\section{The Heinz and Hopf constants and the main results}

The constants defined by \textbf{Heinz and Hopf } are related to the maximum possible value of the absolute value of Gaussian curvature for a minimal graph over a disk.

\subsection{Heinz's Constant}

In 1952, E. Heinz proved that for any minimal graph over a disk of radius $R$, the Gaussian curvature is bounded on the point above the center of the disk by an absolute constant. Specifically, there exists a constant $C_0$ such that for any minimal graph $\mathbf{f}: D_R \to \mathbb{R}$, we have:
$$|\mathcal{K}(\xi)| \leq \frac{C_0}{R^2}$$
where $\xi$ is the point above the center of the disk $D_R$. For the unit disk ($R=1$), the inequality simplifies to $|\mathcal{K}(\xi)| \leq C_0$. The main problem was to find the constant $$c_0=\inf\{C_0, \emph{so that the previous inequality does hold}\}.$$ We call this constant \textbf{Heinz's constant}. The conjectured value for this constant, which was later proven by the author and Melentijevi\'c in \cite{kalaj2025gaussian}, is $c_0 = \frac{\pi^2}{2}$. This value is achieved for the ``Scherk's surface,'' a classical example of a minimal surface with a particular periodic structure.

Finn and Osserman in \cite{FinnOsserman1964} proved that $c_0=\pi^2/2$ under condition that the gradient vanishes at origin. They also proved this theorem \begin{proposition}\label{proposs} \textit{Given any number \( c < \pi^2/2 \), there exists a minimal surface \( \mathbf{f}(x,y) \) defined over \( x^2 + y^2 < 1 \) whose gradient vanishes at the origin and whose Gauss curvature at the origin satisfies \( |K| > c  \).}
\end{proposition}
In this paper we will improve and generalize Proposition~\ref{proposs} by proving
 \begin{theorem} \label{inth}
Let
\[
S_+^2 = \{(\mathrm{x},\mathrm{y},\mathrm{z}) \in \mathbb{R}^3 : \mathrm{x}^2 + \mathrm{y}^2 + \mathrm{z}^2 = 1, \quad \mathrm{z} \geq 0 \}
\]
denote the closed upper hemisphere. For every unit vector
\(
n=(\mathrm{x},\mathrm{y},\mathrm{z}) \in S_+^2
\)
and every
\[
0 \leq \kappa < \mathbf{c}_0(\frac{\mathrm{x}+i\mathrm{y}}{1 - \mathrm{z}})\le \frac{\pi^2}{2},
\]
where \( \mathbf{c}_0 \) is defined in \eqref{finoser0},
there exists a minimal graph
\[
\Sigma = \{ (u,v,\mathbf{f}(u,v)) : (u,v) \in \mathbb{D} \},
\]
over the unit disk \( \mathbb{D} \subset \mathbb{R}^2 \), $\mathbf{f}(0,0)=0$, such that:

\begin{itemize}
    \item The unit normal at the origin \( 0 \) equals \( n \),
    \item The surface bends in the coordinate directions at \( 0 \),
    \item The absolute value of the Gaussian curvature at \( 0 \) satisfies
    \[
    |K_\Sigma(0)| = \kappa.
    \]
\end{itemize}

Moreover, the condition
\(
\kappa < c_1(w)
\)
is sharp. In particular, for $n=(0,0,1)$, $\mathbf{c}_0(0)=\pi^2/2$ and thus, this result recovers Proposition~\ref{proposs}.

\end{theorem}

\subsection{Hopf's Constant}

Shortly after Heinz's work, E. Hopf introduced a related problem.  He considered a different quantity to bound: the product of the mofulus of the Gaussian curvature and the square of the norm of the normal vector. For a minimal graph of a function $\mathbf{f}(u,v)$, where $(u,v)\in D$, the Gaussian unit normal is given by:
$$ \mathcal{N} = \frac{(-\mathbf{f}_u, -\mathbf{f}_v, 1)}{\sqrt{1+\mathbf{f}_u^2+\mathbf{f}_v^2}} $$
The norm of the normal vector is $W = \sqrt{1+\mathbf{f}_u^2+\mathbf{f}_v^2}$. The quantity of interest is:
$$W^2 |\mathcal{K}| = W^2 \frac{|\mathbf{f}_{uu}\mathbf{f}_{vv}-\mathbf{f}_{uv}^2|}{W^4} = \frac{|\mathbf{f}_{uu}\mathbf{f}_{vv}-\mathbf{f}_{uv}^2|}{W^2}$$
Hopf sought the best constant $c_1$ in the inequality $W^2|\mathcal{K}| \leq c_1$ for all minimal graphs over the unit disk. This constant is now known as the \textbf{Hopf constant}.
The value of this constant is also conjectured to be $c_1 = \frac{\pi^2}{2}\approx4.9348$. Evidence for this conjecture was provided by Finn and Osserman for minimal graphs with a horizontal plane at the origin, and by Nitsche in \cite{zbMATH03431423} for symmetric minimal graphs. Nitsche also gave estimates for general minimal graphs in \cite{Nitsche1965}, showing that $c_1 \le 7.678447$.

A further result by Hall in \cite{Hall1982} implies a tighter bound,  $c_1 \le \frac{16\pi^2}{27} \approx 5.84865$ (see the proof of Theorem~\ref{seconth}). In this paper, we improve this bound by proving \begin{theorem} \label{seconth}
We have that ${c}_1<\frac{2\pi ^2}{27} \left(2+\sqrt{31}\right) \approx 5.53265.$
\end{theorem}

\section{Preliminaries and the proof of main results}
\subsection{Riemann mapping theorem (RMT)  for  harmonic mappings} \begin{theorem}[Harmonic mapping analogue with Beltrami condition]\cite{HengartnerSchober1986}\label{hes}
Let \( \Omega \subset \mathbb{C} \) be a simply connected domain, and let \( f = h + \overline{g} \) be a complex-valued harmonic mapping from the unit disk \( \mathbb{D} \) onto \( \Omega \), where \( h \) and \( g \) are analytic in \( \mathbb{D} \).

Assume that:
\begin{itemize}
    \item \( f \) is sense-preserving and univalent,
    \item Beltrami coefficient \( \omega = \frac{g'}{h'} \) satisfies \( \|\omega\|_\infty < k < 1 \),
    \item and the normalization \( f(0) = z_0 \in \Omega \), \( h'(0) > 0 \) holds.
\end{itemize}

Then such a harmonic mapping \( f \) exists and is unique.
\end{theorem}

If $\|\omega(z)\|_{\infty}=1$, then the RMT is no longer true and the following more general Riemann mapping theorem (GRMT) of Hengartner and Schober  holds.

\begin{theorem}\cite{HengartnerSchober1986}
Let $\Omega$ be a bounded simply connected domain with locally connected boundary, and let $w_{0}$ be a fixed point of $\Omega$. Also, let $\omega$ be an analytic function of $\mathbb{D}$ with $\omega(\mathbb{D})\subset\mathbb{D}$. Then there exists a univalent harmonic mapping $f$ of $\mathbb{D}$ that satisfies the following properties:

\begin{itemize}
    \item[(a)] $f$ is a solution of $\overline{f}_z(z)=\omega(z) f_{ z}(z)$.

    \item[(b)] $f$ maps $\mathbb{D}$ into $\Omega$ with $f(0)=w_{0}$ and $f_{z}(0)>0$.

    \item[(c)] The unrestricted limit $f^{*}(e^{it})=\lim_{z\to e^{it}}f(z)$ exists and belongs to $\partial\Omega$ for all but a countable subset $E$ of the unit circle $\mathbb{T}=\partial\mathbb{D}$.

    \item[(d)] The one-sided limits $\lim_{s\to t^{+}}f^{*}(e^{is})$ and $\lim_{s\to t^{-}}f^{*}(e^{is})$, through values $e^{is}\not\in E$, exist, belong to $\partial\Omega$, and are equal if $e^{it}\not\in E$ and distinct otherwise.

    \item[(e)] The cluster set of $f$ at $e^{it}\in E$ is the straight line segment joining the one-sided limits $\lim_{s\to t^{+}}f^{*}(e^{is})$ and $\lim_{s\to t^{-}}f^{*}(e^{is})$.
\end{itemize}
\end{theorem}

We remark that if $\|\omega\|_{\infty}<1$, then the GRMT reduces to the RMT.
Then D. {Bshouty} and A. {Lyzzaik} and A. {Weitsman} in \cite{zbMATH05159460}, extended the latter uniqueness result of Hengartner and Schober in \cite{HengartnerSchober1986} for the case where the dilatation $\omega$ is a finite Blaschke product. Their result states as follows.

\begin{theorem}\cite{zbMATH05159460}
Let

\begin{enumerate}
    \item[(a)] $\Omega$ be a bounded convex domain,

    \item[(b)] $w_0$ be a fixed point of $\Omega$,

    \item[(c)] $\omega$ be a finite Blaschke product of order $N=1,2,\cdots$, and

    \item[(d)] $f$ be a function defined as in Theorem A.
\end{enumerate}

Then $f$ is unique.
\end{theorem}

In order to formulate Theorem~\ref{prejprej}, which is the main tool in proving our results we begin by the following definition.
\begin{definition}
In Euclidean geometry, a bicentric quadrilateral is a convex quadrilateral that has both an incircle and a circumcircle.
Assume that  $Q=Q(a,b,c,d)$  is a bicentric quadrilateral inscribed in the unit disk $\mathbb{D}$. A minimal graph  $S=\{(u,v,\mathbf{f}(u,v)), (u,v)\in Q\}$ over the quadrilateral $Q$ is called a Scherk type surface if it satisfies $\mathbf{f}(u,v)\rightarrow + \infty$ when $(u,v) \rightarrow \zeta\in(a,b)\cup (c,d)$ and $\mathbf{f}(u,v)\rightarrow -\infty$ when $(u,v) \rightarrow \zeta\in(b,c)\cup (a,d)$ .
\end{definition}

\begin{theorem}\label{prejprej}\cite{kalaj2025gaussian, kalaj2021curvature}
For every $w\in\D$, there exist a bicentric quadrilateral $Q(a_0, a_1,a_2,a_3)$ with vertices at the unit circle $\mathbf{T}$  and then there is a harmonic unique mapping $f=f^w$ of the unit disk onto $Q(a_0,a_1,a_2,a_3)$ that solves the Beltrami equation \begin{equation}\label{beleq}\bar f_z(z) = \left(\frac{w+\frac{\imath \left(1-w^4\right) z}{\left|1-w^4\right|}}{1+\frac{\imath\overline{w} \left(1-w^4\right) z}{\left|1-w^4\right|}}\right)^2 f_z(z),\end{equation} $|z|<1$ and satisfies the initial condition $f(0)=0$, $f_z(0)>0$. It also defines a Scherk's type minimal surface $S^\diamond: \zeta=\mathbf{f}^\diamond(u,v)$ over the  quadrilateral $Q(a_0,a_1,a_2,a_3)$, with the centre $\mathbf{w}=(0,0,0)$ so that its Gaussian normal is stereographic projection of $w$: $$\mathbf{n}^\diamond_{\mathbf{w}}=-\frac{1}{1+|w|^2}(2\Im w, 2\Re w, -1+|w|^2),$$ and $D_{uv}\mathbf{f}^\diamond(0,0)=0$. Moreover, every other non-parametric minimal surface $S:$ $z=\mathbf{f}(u,v)$ over the unit disk, with a centre  $\mathbf{w}$, with $\mathbf{n}_{\mathbf{w}}=\mathbf{n}^\diamond_{\mathbf{w}}$ and $D_{uv}\mathbf{f}(0,0)=0$ satisfies the sharp inequality $$|\mathcal{K}_{S}(\mathbf{w})|<|\mathcal{K}_{S^\diamond}(\mathbf{w})|,$$ or what is the same
$$W^2_{S}|\mathcal{K}_{S}(\mathbf{w})|<W^2_{S^\diamond}|\mathcal{K}_{S^\diamond}(\mathbf{w})|.$$

\end{theorem}
\begin{remark}
The unique harmonic mapping $f^w$ of the unit disk onto the quadrilateral $Q(w)$ determined in the previous theorem determines the functions $\mathbf{c}_0(w)$ and  $\mathbf{c}_1(w)$:
\begin{equation}\label{finoser0}\mathbf{c}_0(w):=-\mathcal{K}_{S^\diamond}(\mathbf{w})=\frac{4(1-|w|^2)^2 }{\left(1+|w|^2\right)^4 |f^w_z(0)|^2},\end{equation}
 \begin{equation}\label{finoser}\mathbf{c}_1(w):=-W^2_{S^\diamond}\mathcal{K}_{S^\diamond}(\mathbf{w})=\frac{4 }{\left(1+|w|^2\right)^2 |f^w_z(0)|^2}.\end{equation}
It  follows from Theorem~\ref{prejprej} and Theorem~\ref{inth} (i.e. from \eqref{finoser0} and \eqref{finoser}) that the Heinz  and Hopf constant can be determined as follows
 \begin{equation} \label{heinz}c_0 = \sup_{w} \frac{4 \left(1-|w|^2\right)^2}{\left(1+|w|^2\right)^4 |f^w_z(0)|^2}\end{equation} and  \begin{equation} \label{hopf}c_1 = \sup_{w} \frac{4 }{\left(1+|w|^2\right)^2 |f^w_z(0)|^2}.\end{equation}
 \end{remark}

\begin{proof}[Proof of Theorem~\ref{inth}]  Let  $$\mu_k(z)=\left(\frac{w+k\frac{\imath \left(1- w^4\right) z}{\left|1-w^4\right|}}{1+k\frac{\imath\overline{w} \left(1-w^4\right) z}{\left|1-w^4\right|}}\right)^2,$$ where $k\in[0,1)$. Then $\|\mu_k\|_\infty\le \left(\frac{|w|+k}{1+|w|k}\right)^2<1$.
Let $f^k$ be a univalent solution of the Beltrami equation \begin{equation}\label{beleqk}\bar f_z(z) = \mu_k(z) f_z(z),\end{equation} $|z|<1$, mapping the unit disk onto itself such that $f^k(0)=0$ and $\partial f^k(0)>0$. This mapping exists and is unique (Theorem~\ref{hes}).
Note that $$\mathbf{q}=\sqrt{\mu_k}=\frac{w+k\frac{\imath \left(1- w^4\right) z}{\left|1-w^4\right|}}{1+k\frac{\imath\overline{w} \left(1-w^4\right) z}{\left|1-w^4\right|}}$$ is indeed  the Gauss map of the minimal surface, after an adequate
stereographic projection. See e.g.  \cite[p.~169]{Duren2004}). Hence, for $f=f^k$, $\mathbf{p}=f^k_z$,
\[
\varpi(z) = \left(\Re f(z), \Im f(z), \Im \int_0^z 2 \mathbf{p}(\zeta) \mathbf{q}(\zeta) d\zeta \right),\quad z\in\D.
\]
represent a conformal minimal parameterization of a minimal graph $S_k$.
The curvature $\mathcal{K}$ of the minimal graph $S_k$ at $\varpi(z)$ is expressed in terms
of the Enneper--Weierstrass parameters $(\mathbf{p},\mathbf{q})$, by
\begin{equation}\label{eq:curvatureformula}
\mathcal{K}(z) = - \frac{4|\mathbf{q}'(z)|^2}{|\mathbf{p}(z)|^2(1 + |\mathbf{q}(z)|^2)^4},
\end{equation}
where $\mathbf{p}=f_z$ and $\omega=\mathbf{q}^2=\overline{f_{\bar z}}/f_z$. (See Duren \cite[p.\ 184]{Duren2004}.)
Given
\[
\mathbf{q}(z) = \frac{w + k\frac{i(1 - w^4)z}{|1 - w^4|}}{1 + k\frac{i \overline{w}(1 - w^4)z}{|1 - w^4|}},
\]
we compute
\[
\frac{|\mathbf{q}'(0)|^2}{(1 + |\mathbf{q}(0)|^2)^4}.
\]

Since \( \mathbf{q}(0) = w \), we have \( |\mathbf{q}(0)| = |w| \). Also,
\[
\mathbf{q}'(0) = k \cdot \frac{i(1 - w^4)}{|1 - w^4|} (1 - |w|^2),
\]
so
\[
|\mathbf{q}'(0)|^2 = k^2 (1 - |w|^2)^2.
\]

Thus,
\[
{
\frac{|\mathbf{q}'(0)|^2}{(1 + |\mathbf{q}(0)|^2)^4} = \frac{k^2 (1 - |w|^2)^2}{(1 + |w|^2)^4}
}.
\]

Hence \begin{equation}\label{mk0}\mathcal{K}_k(0)=-\frac{k^2 (1 - |w|^2)^2}{|f^k_z(0)|^2(1 + |w|^2)^4}.\end{equation}

 As $k\to 1^-$, $f^k$  tends to a harmonic diffeomorphism $f^\ast$ of the unit disk onto a bicentric quadrilateral $Q$. Then such a mapping $f_\ast$ satisfies the conditions $f^\ast_z(0)>0 $, $f^\ast(0)=0$ and has the dilatation $$\mu(z) = \left(\frac{w+\frac{\imath \left(1- w^4\right) z}{\left|1-w^4\right|}}{1+\frac{\imath\overline{w} \left(1-w^4\right) z}{\left|1-w^4\right|}}\right)^2.$$

Moreover as it will be proved in the sequel (see Lemma~\ref{lelepa}), the Gaussian curvature  $-\mathcal{K}_k(\mathbf{w})$ at $\mathbf{w}$ is a continuous function on $k$, satisfying the conditions $-\mathcal{K}_0(\mathbf{w})=0$ and  $$\mathcal{K}_1(\mathbf{w})=\mathcal{K}_{S^\diamond}(\mathbf{w}).$$ This finishes the proof up to the Lemma~\ref{lelepa} below.
\end{proof}

 \begin{lemma}\label{lelepa}
 The function $P(k) = \mathcal{K}_k(\mathbf{w})$ is continuous.
 \end{lemma}
 The previous lemma follows from \eqref{mk0} and the following lemma
 \begin{lemma} Let $f^t$ be a family of harmonic mappings of the unit disk onto itself solving the Beltrami equation with the Beltrami coefficient $\omega_t(z)$ satisfying the condition $|\omega_t(z)|\le k<1$ and $f^t(0)=0$, $f^t_z(0)>0$. Then if $\omega_t \to \omega_{t_0}$ in $L^\infty(\mathbb{D})$ norm, then $f^t_z(0)\to f^{t_0}_z(0)$. Here $t\in[0,1]$.
\end{lemma}

\begin{proof} Denote for the simplicity $\omega=\omega_t$.
Assume first that $\omega(0) = 0$. Let $\omega(z)$ be an analytic function in the unit disk $\mathbb{D} = \{z \in \mathbb{C} : |z| < 1\}$, satisfying the bound $|\omega(z)| \leq k < 1$ for some constant $k$. We are given that a unique fixed point $\varphi \in L^p(\mathbb{D})$ exists for some $p > 2$, solving the fixed-point equation:
\[
\varphi = T_\omega(\varphi),
\]
where the nonlinear operator $T_\omega$ is defined by
\[
T_\omega(\varphi)(z) = \frac{\overline{\omega(z)}}{z} \cdot (\mathcal{M}\varphi)(z)\left[1 + \overline{z} \cdot (\mathcal{H}\varphi)(z)\right].
\]

The operator $\mathcal{M}$ denotes a certain nonlinear multiplier, while $\mathcal{H}$ is the \textit{Hilbert transform} on $\mathbb{D}$, defined as the complex derivative of a modified \textit{Cauchy transform} $\mathcal{P}$:
\[
(\mathcal{P}\varphi)(z) = -\frac{1}{\pi} \iint_{\mathbb{D}} \left\{ \frac{\varphi(\zeta)}{\zeta - z} + \frac{z\overline{\varphi(\zeta)}}{1 - \overline{\zeta} z} - \frac{\varphi(\zeta)}{2\zeta} + \frac{\overline{\varphi(\zeta)}}{2\overline{\zeta}} \right\} d\xi \, d\eta, \quad \zeta = \xi + i\eta.
\]

Its derivative with respect to $z$ is given by
\[
(\mathcal{P}\varphi)_z(z) = -\frac{1}{\pi} \iint_{\mathbb{D}} \frac{\varphi(\zeta)}{(\zeta - z)^2} \, d\xi \, d\eta
- \frac{1}{\pi} \iint_{\mathbb{D}} \frac{\overline{\varphi(\zeta)}}{(1 - \overline{\zeta} z)^2} \, d\xi \, d\eta.
\]
The first integral is understood in the principal value sense and corresponds to the classical Hilbert transform.

For notational convenience, we write
\[
(\mathcal{P}\varphi)_z = \mathcal{H}\varphi.
\]

As shown in \cite[p.~132]{Duren2004}, the solution of the Beltrami equation in this context is given by
\[
f(z) = z \exp(\mathcal{P}\varphi(z)).
\]
This operator $\mathcal{H}$ has Hilbert norm equal to 1. See \cite[p.~151]{AstalaIwaniecMartin2009} or  \cite{KalajMelentijevicZhu2020}.
The term $\mathcal{M}\varphi$ is defined as:
\[\mathcal{M}\varphi =  \exp\{-2i \text{Im}\{\mathcal{P}\varphi\}\}. \text{}\]

A key property is $|\mathcal{M}\varphi| = 1$. Observe that in this case $\Omega=\mathbb{D}$ so, in notation of \cite{Duren2004} $\Phi(z)=z$.

Our primary objective is to demonstrate that if a sequence of parameters $\omega_t$ converges to $\omega_0$ (i.e., $\omega_t \to \omega_0$ in the $L^\infty(\mathbb{D})$ norm, implying uniform convergence on the unit disk), then their corresponding unique fixed points $\varphi_t$ converge to $\varphi_0$ (i.e., $\varphi_t \to \varphi_0$) in $L^p(\mathbb{D})$.

\subsection{Continuous dependence of the fixed point $\varphi_t$ on $\omega_t$}

To establish the continuous dependence of the fixed point $\varphi_t$ on the parameter $\omega_t$, we analyze the properties of the operator $T(\omega, \varphi)$, which involves examining its continuity and compactness.

\subsubsection{Properties of components}

We begin by observing that the quantity \(\frac{\overline{\omega_t(z)}}{z}\) remains uniformly bounded. Each function \(\omega_t(z)\) is analytic in the unit disk \(\mathbb{D}\), vanishes at the origin, and satisfies \(|\omega_t(z)| \leq k < 1\). Defining \(h_t(z) = \frac{\omega_t(z)}{z}\), the generalized Schwarz Lemma implies \(|h_t(z)| \leq k\) for all \(z \in \mathbb{D}\), including \(z=0\) where \(|h_t(0)| = |\omega_t'(0)| \leq k\). Therefore, \(h_t\) is uniformly bounded in \(L^\infty(\mathbb{D})\), and so is its complex conjugate over \(z\), yielding
\[
\left\|\frac{\overline{\omega_t(z)}}{z}\right\|_{L^\infty} \leq k < 1 \quad \text{for all } t.
\]

Next, consider the key operators involved. The Hilbert transform \(\mathcal{H}\) acts boundedly on \(L^p(\mathbb{D})\) for \(1 < p < \infty\), with norm depending continuously on \(p\); in particular, its norm on \(L^2\) is 1. The operator \(P\), as defined in the context, is compact from \(L^p(\mathbb{D})\) to itself when \(p > 2\), and it maps into the Sobolev space \(W^{1,p}(\mathbb{D})\). By the Sobolev embedding theorem, this image lies in \(C(\overline{\mathbb{D}})\), so \(P\varphi\) defines a continuous function on the closed disk whenever \(\varphi \in L^p\) for \(p > 2\).

The operator \(\mathcal{M}\varphi\) incorporates an exponential composed with \(P\varphi\), and thus inherits continuity due to the continuity of \(P\varphi\). Given the modulus constraint \(|\mathcal{M}\varphi| = 1\), this operator maps bounded sets in \(L^p(\mathbb{D})\) into bounded sets in \(L^\infty(\mathbb{D})\) via a continuous (though nonlinear) transformation.

Let us now turn to the central object of interest, the operator \(T(\omega, \varphi)\), and analyze its continuity and compactness. The goal is to establish joint continuity in both variables and compactness for fixed \(\omega\).

To show joint continuity, suppose \(\omega_t \to \omega_0\) in \(L^\infty(\mathbb{D})\) and \(\varphi_t \to \varphi_0\) in \(L^p(\mathbb{D})\). We estimate the difference
\[
\|T(\omega_t, \varphi_t) - T(\omega_0, \varphi_0)\|_{L^p}
\]
using the triangle inequality:
\[
\le \|T(\omega_t, \varphi_t) - T(\omega_t, \varphi_0)\|_{L^p} + \|T(\omega_t, \varphi_0) - T(\omega_0, \varphi_0)\|_{L^p}.
\]

For the second term, write \(T(\omega, \varphi) = \frac{\overline{\omega}}{z} A(\varphi)\), where \(A(\varphi) = \mathcal{M}\varphi [1 + \overline{z}(\mathcal{H}\varphi)]\). Then
\[
\|T(\omega_t, \varphi_0) - T(\omega_0, \varphi_0)\|_{L^p} = \left\| \left(\frac{\overline{\omega_t(z)}}{z} - \frac{\overline{\omega_0(z)}}{z}\right) A(\varphi_0) \right\|_{L^p}
\]
\[
\le \left\|\frac{\overline{\omega_t(z)}}{z} - \frac{\overline{\omega_0(z)}}{z}\right\|_{L^\infty} \cdot \|A(\varphi_0)\|_{L^p},
\]
which tends to zero since \(\omega_t \to \omega_0\) in \(L^\infty\) and \(A(\varphi_0)\) is fixed in \(L^p\).

For the first term, we express the difference using
\[
T(\omega, \varphi) = \frac{\overline{\omega}}{z} \cdot \mathcal{M}\varphi \cdot (1 + \overline{z} \mathcal{H}\varphi).
\]
Applying the product rule \((ab - cd) = (a - c)b + c(b - d)\), we write
\[
M(\varphi_t)N(\varphi_t) - M(\varphi_0)N(\varphi_0) = (M(\varphi_t) - M(\varphi_0))N(\varphi_t) + M(\varphi_0)(N(\varphi_t) - N(\varphi_0)),
\]
where \(M(\varphi) = \mathcal{M}\varphi\), and \(N(\varphi) = 1 + \overline{z} \mathcal{H}\varphi\).

The first term vanishes since \(M(\varphi_t) \to M(\varphi_0)\) uniformly in \(L^\infty\) and \(N(\varphi_t)\) remains bounded in \(L^p\). The second vanishes because \(N(\varphi_t) \to N(\varphi_0)\) in \(L^p\), due to the continuity of \(\mathcal{H}\). Thus, both contributions tend to zero, establishing the joint continuity of \(T(\omega, \varphi)\).

For compactness, recall that \(P\) is compact from \(L^p\) into itself (for \(p > 2\)), and its image lies in \(C(\overline{\mathbb{D}})\). The operator \(\mathcal{M}\), built from continuous operations on \(P\varphi\), thus maps bounded subsets of \(L^p\) into relatively compact subsets of \(L^\infty\). Composing this with bounded multipliers such as \(\frac{\overline{\omega}}{z}\) and the Hilbert transform yields that \(T(\omega, \cdot)\) is compact on \(L^p(\mathbb{D})\).

We now turn to the continuous dependence of the fixed point \(\varphi_t\) on \(\omega_t\). Suppose, for contradiction, that \(\varphi_t \not\to \varphi_0\) in \(L^p\). Then there exists \(\epsilon_0 > 0\) and a subsequence (still denoted \(\varphi_t\)) such that \(\|\varphi_t - \varphi_0\|_{L^p} \ge \epsilon_0\) for all \(t\).

Since the operator \(T(\omega_t, \cdot)\) is compact and each \(\varphi_t = T(\omega_t, \varphi_t)\), the sequence \(\{\varphi_t\}\) is uniformly bounded. By compactness, it has a convergent subsequence \(\varphi_{t_k} \to \varphi^*\). Passing to the limit in the fixed-point equation yields
\[
\varphi^* = \lim_{k \to \infty} T(\omega_{t_k}, \varphi_{t_k}) = T(\omega_0, \varphi^*),
\]
so \(\varphi^*\) is a fixed point of \(T(\omega_0, \cdot)\). Uniqueness then implies \(\varphi^* = \varphi_0\), contradicting the assumption \(\|\varphi_{t_k} - \varphi_0\| \ge \epsilon_0\). Hence, \(\varphi_t \to \varphi_0\) in \(L^p(\mathbb{D})\).

Having established that the fixed point \(\varphi_t\) depends continuously on the parameter \(\omega_t\), we now apply this to show the continuous dependence of the derivative \(f_z(0)\) of the associated harmonic mapping.

\vspace{1em}
\noindent
The harmonic map \(f(z)\) is related to \(\varphi(z)\) through the operator \(P\). In particular, when \(\omega(0) = 0\), the derivative at the origin satisfies \(f_z(0) = e^{(P\varphi)(0)}\). The value \((P\varphi)(0)\) is given explicitly by
\[
(P\varphi)(0) = -\frac{1}{\pi} \iint_{\mathbb{D}} \left( \frac{\varphi(\zeta)}{2\zeta} + \frac{\overline{\varphi(\zeta)}}{2\overline{\zeta}} \right) d\xi \, d\eta,
\]
which yields a real value due to the symmetry of the integrand. Consequently, \(f_z(0)\) is a positive real number.

When \(\omega_t(0) = 0\), the continuity of \(f_z^t(0)\) in \(t\) follows directly. Since we already know \(\varphi_t \to \varphi_0\) in \(L^p(\mathbb{D})\), and the map \(\varphi \mapsto (P\varphi)(0)\) is a continuous linear functional on \(L^p(\mathbb{D})\), it follows that \((P\varphi_t)(0) \to (P\varphi_0)(0)\). The continuity of the exponential function then ensures that
\[
f_z^t(0) = e^{(P\varphi_t)(0)} \to e^{(P\varphi_0)(0)} = f_z^0(0).
\]

In the more general case where \(\omega_t(0) \neq 0\), we normalize the Beltrami coefficient using the automorphism
\[
\tilde{\omega}_t(z) = \frac{\omega_t(z) - \omega_t(0)}{1 - \overline{\omega_t(0)} \omega_t(z)},
\]
which satisfies \(\tilde{\omega}_t(0) = 0\) and retains the bound \(|\tilde{\omega}_t(z)| < 1\). This normalization transforms the original harmonic mapping \(f(z)\) into another mapping \(\tilde{f}(z)\), corresponding to \(\tilde{\omega}_t(z)\), for which the origin condition is satisfied.

This leads to the representation
\[
f(z) = \frac{\tilde{f}(z) + \omega(0) \, \overline{\tilde{f}(z)}}{1 - |\omega(0)|^2},
\]
where \(\tilde{f}\) solves the Beltrami equation \(\tilde{f}_{\overline{z}} = \tilde{\omega} \tilde{f}_z\) with \(\tilde{f}(0) = 0\) and \(\tilde{f}_z(0) > 0\).

Differentiating the formula for \(f(z)\) yields
\[
f_z(z) = \frac{1}{1 - |\omega(0)|^2} \left( \tilde{f}_z(z) + \omega(0) \, \overline{\tilde{f}_{\overline{z}}(z)} \right).
\]
At \(z = 0\), we know \(\tilde{f}_{\overline{z}}(0) = \tilde{\omega}(0)\tilde{f}_z(0) = 0\), hence
\[
f_z(0) = \frac{\tilde{f}_z(0)}{1 - |\omega(0)|^2}.
\]

Now, suppose \(\omega_t \to \omega_0\) in \(L^\infty(\mathbb{D})\). Then by continuity of evaluation at the origin, \(\omega_t(0) \to \omega_0(0)\). It follows that the denominator \(1 - |\omega_t(0)|^2\) converges to \(1 - |\omega_0(0)|^2\), a positive number since \(|\omega(0)| < 1\).

Moreover, since the map \(\omega \mapsto \tilde{\omega}\) is continuous with respect to the \(L^\infty\)-norm, we have \(\tilde{\omega}_t \to \tilde{\omega}_0\) in \(L^\infty(\mathbb{D})\). From the previous case (where the normalized Beltrami coefficient vanishes at the origin), we know that \(\tilde{f}_z^t(0) \to \tilde{f}_z^0(0)\). Combining these facts, we obtain
\[
f_z^t(0) = \frac{\tilde{f}_z^t(0)}{1 - |\omega_t(0)|^2} \to \frac{\tilde{f}_z^0(0)}{1 - |\omega_0(0)|^2} = f_z^0(0),
\]
as \(t \to 0\).

Hence, regardless of whether \(\omega_t(0) = 0\) or not, the quantity \(f_z^t(0)\) depends continuously on \(\omega_t\) in the topology of \(L^\infty(\mathbb{D})\).

\end{proof}


\begin{proof}[Proof of Theorem~\ref{seconth}] By Hall result \cite{Hall1982}
\[
|a_1|^2 + \frac{3\sqrt{3}}{\pi}|a_0|^2 + |b_1|^2 \geq \frac{27}{4\pi^2}.
\]
Since $a_0=0$ and  $b_1=(\mathbf{q}(0))^2 a_1$, for $w=\mathbf{q}(0)$, we have
\[
|a_1|^2 + |w|^4 |a_1|^2 \geq \frac{27}{4\pi^2}.
\]
Then
\[
|a_1|^2 \geq \frac{27}{4\pi^2 (1 + |w|^4)},
\]
so
\[
c_1 = \sup_{|w| < 1} \frac{4}{(1 + |w|^2)^2 |a_1|^2}
\leq \frac{16\pi^2}{27} \cdot \sup_{|w| < 1} \frac{1 + |w|^4}{(1 + |w|^2)^2}.
\]
Letting \( x = |w|^2 \), we obtain
\[
\sup_{0 \le x < 1} \frac{1 + x^2}{(1 + x)^2} = 1,
\]
and therefore
\[
c_1 \le  \frac{16\pi^2}{27}.
\]

Further from \eqref{finoser}, for a fixed unit normal determined by $w$ and by the condition $\mathbf{f}_{u,v}(0,0)=0$, we have $$\mathcal{K}(\mathbf{w})\le \mathbf{c}_0(w)=\frac{4 \left(1-|w|^2\right)^2}{\left(1+|w|^2\right)^4 |f^w_z(0)|^2}.$$

Thus we obtain that \begin{equation}\label{merku}
\mathcal{K}(\mathbf{w})\le \frac{16\pi^2}{27}\frac{1+|w|^4}{(1+|w|^2)^2}.
\end{equation}

According to \eqref{heinz} and \eqref{hopf} and by using the proved fact that $c_0=\pi^2/2$ we obtain $$\mathbf{c}_1(w)\le \frac{\pi^2}{2}\frac{(1+|w|^2)^2}{(1-|w|^2)^2}.$$
Thus combining by \eqref{merku}, we obtain  $$\mathbf{c}_1(w) < \min \left\{\frac{\pi^2}{2}\frac{(1+|w|^2)^2}{(1-|w|^2)^2},\frac{16\pi^2}{27}\frac{1+|w|^4}{(1+|w|^2)^2}\right\}.$$

We also have $$c_1=\sup_{|w|<1} \mathbf{c}_1(w).$$
We want to find the maximum of the function:
$$f(|w|) = \min\left(\frac{\pi^2}{2}\frac{(1+|w|^2)^2}{(1-|w|^2)^2}, \frac{16\pi^2}{27}\frac{1+|w|^4}{(1+|w|^2)^2}\right) \quad \text{for } |w| \in [0, 1].$$

Let $x = |w|^2$. Since $|w| \in [0, 1]$, we have $x \in [0, 1]$. The function can be rewritten as:
$$g(x) = \min\left(f_1(x), f_2(x)\right)$$
where
$$f_1(x) = \frac{\pi^2}{2}\frac{(1+x)^2}{(1-x)^2}$$
$$f_2(x) = \frac{16\pi^2}{27}\frac{1+x^2}{(1+x)^2}.$$
It is easy to prove that  $f_1(x)$ is increasing and $f_2(x)$ is decreasing on $[0, 1)$, so the maximum  of $g(x)$ occurs at their intersection point.

Set $f_1(x) = f_2(x)$:
$$\frac{\pi^2}{2}\frac{(1+x)^2}{(1-x)^2} = \frac{16\pi^2}{27}\frac{1+x^2}{(1+x)^2}.$$
Divide by $\pi^2$ and cross-multiply:
$$27(1+x)^4 = 32(1+x^2)(1-x)^2$$
This is equivalent with
$$5x^4 - 172x^3 - 98x^2 - 172x + 5 = 0.$$

 Since $x=0$ is not a solution, we can divide by $x^2$:
$$5x^2 - 172x - 98 - \frac{172}{x} + \frac{5}{x^2} = 0$$
Group terms:
$$5\left(x^2 + \frac{1}{x^2}\right) - 172\left(x + \frac{1}{x}\right) - 98 = 0$$
Let $y = x + \frac{1}{x}$. Then $y^2 = x^2 + 2 + \frac{1}{x^2}$, so $x^2 + \frac{1}{x^2} = y^2 - 2$.
Substitute into the equation:
$$5y^2 - 172y - 108 = 0.$$
Then
$$y_{1,2} = \frac{172 \pm 32\sqrt{31}}{10} = \frac{86 \pm 16\sqrt{31}}{5}$$
Since $x \in [0, 1]$, $x$ is real and positive. For $x \in (0, 1]$, $y = x + \frac{1}{x} \ge 2$.
We choose the positive root:
$$y = \frac{86 + 16\sqrt{31}}{5}.$$

Substitute $y$ back into a simplified form of $f_1(x)$. We can rewrite $f_1(x)$ in terms of $y$:
$$f_1(x) = \frac{\pi^2}{2}\frac{(1+x)^2}{(1-x)^2} = \frac{\pi^2}{2}\frac{1+2x+x^2}{1-2x+x^2}$$
Divide numerator and denominator by $x$:
$$f_1(x) = \frac{\pi^2}{2}\frac{\frac{1}{x}+2+x}{\frac{1}{x}-2+x} = \frac{\pi^2}{2}\frac{\left(x+\frac{1}{x}\right)+2}{\left(x+\frac{1}{x}\right)-2} = \frac{\pi^2}{2}\frac{y+2}{y-2}$$
The maximum value is:
$$g_{max} = \frac{\pi^2}{2}\frac{\frac{96 + 16\sqrt{31}}{5}}{\frac{76 + 16\sqrt{31}}{5}} = \frac{\pi^2}{2} \frac{96 + 16\sqrt{31}}{76 + 16\sqrt{31}}=  2\pi^2 \frac{2+\sqrt{31}}{27}.$$

\end{proof}
\subsection*{Acknowledgements}
The author is partially supported by the Ministry of Education, Science and Innovation of Montenegro through the grant \emph{Mathematical Analysis, Optimization and Machine Learning}.
{\bibliographystyle{abbrv} \bibliography{references}}


\noindent David Kalaj

\noindent University of Montenegro, Faculty of Natural Sciences and Mathematics, 81000, Podgorica, Montenegro

\noindent e-mail: {\tt davidk@ucg.ac.me}

\end{document}